\documentclass[letterpaper, 11 pt]{article}
\usepackage{fullpage,amsthm, amsmath, amssymb, amsfonts,epsfig}
\usepackage[margin=10pt,font=small,labelfont=bf,
labelsep=endash]{caption}

\newcommand{\RR}{\mathbb{R}}

\DeclareCaptionLabelFormat{andtable}%
{#1.#2 \& \tablename.\thetable}

\newcommand{\newword}[1]{\textbf{\emph{#1}}}

\newcommand{\HH}{\mathcal{H}}
\newcommand{\CCC}{\mathcal{C}}
\newcommand{\KK}{\mathcal{K}}

\newcommand{\M}{\mathcal{M}}
\newcommand{\I}{\mathcal{I}}

\newcommand{\J}{\mathcal{J}}

\newcommand{\A}{\mathcal{A}}

\def\M{\mathcal{M}}
\def\G{\mathcal{G}}

\newtheorem{conj}{Conjecture}[section]

\newtheorem{theorem}[conj]{Theorem}

\newtheorem{proposition}[conj]{Proposition}
\newtheorem{Lemma}[conj]{Lemma}
\newtheorem{lemma}[conj]{Lemma}

\theoremstyle{definition}
\newtheorem{definition}[conj]{Definition}

\newtheorem{remark}[conj]{Remark}

\begin{document}

\title{Generalized permutohedra, h-vector of cotransversal matroids and pure O-sequences}
\author{Suho Oh}
\date{}
\maketitle

\begin{abstract}
Stanley has conjectured that the h-vector of a matroid complex is a pure O-sequence. We will prove this for cotransversal matroids by using generalized permutohedra. We construct a bijection between lattice points inside a $r$-dimensional convex polytope and bases of a rank $r$ transversal matroid.
\end{abstract}

\section{Introduction}

Matroids, simplicial complexes and their h-vectors are all interesting objects that are of great interest in algebraic combinatorics and combinatorial commutative algebra. An \newword{order ideal} is a finite collection $X$ of monomials such that, whenever $M \in X$ and $N$ divides $M$, then $N \in X$. If all maximal monomials of $X$ have the same degree, then $X$ is pure. A pure $O$-sequence is the vector, $h=(h_0=1,h_1,...,h_t)$, counting the monomials of $X$ in each degree. The following conjecture by Stanley has motivated a great deal of research on h-vectors of matroid complexes:

\begin{conj}
The h-vector of a matroid is a pure O-sequence. 
\end{conj}

The above conjecture has been proven for cographic matroids by both Merino \cite{Me} and Chari \cite{Ch2}. It also has been proven for lattice-path matroids by Schweig \cite{Sch}. Lattice path matroids are special cases of cotransversal matroids, and we will prove the conjecture for cotransversal matroids. We would also like to note that there has been plenty of interesting results related to this conjecture: \cite{BMMNZ},\cite{Ch},\cite{HS},\cite{Hi},\cite{Pr},\cite{Sto},\cite{TSZ}.

We prove the conjecture for cotransversal matroids by associating a polytope to each cotransversal matroid. The lattice points inside this polytope will be in bijection with bases of the matroid, and will naturally induce a pure order ideal we are looking for. 

In section 2, we will go over properties of transversal matroids. In section 3, we go over properties of generalized permutohedra. In section 4, we show a connection between transversal matroids and generalized permutohedra. In section 5, we prove our main result.


\medskip

\textbf{Acknowledgment} The author would like to thank Alexander Postnikov, Richard Stanley, Criel Merino, Fabrizio Zanello and Hwanchul Yoo for useful discussions. The author would also like to thank Jose Soto for his advice on transversal matroids, and David Speyer for his advice on fine mixed subdivisions.

\section{Preliminaries on matroids}
\label{sec:trans}

In this section, we will provide some notations and tools on transversal matroids that we are going to use throughout the paper. We will assume basic familiarity with matroid theory. Throughout the paper, unless stated otherwise, a matroid $\M$ will be a rank $r$ matroid over the ground set $[\bar{n}]:=\{\bar{1} < \cdots < \bar{n}\}$.


An element $i$ of a base $B$ is \newword{internally active} if $(B \setminus \{i\} ) \cup \{j\}$ is not a base for any $j < i$. An element $e \not \in B$ is \newword{externally active} if $(B \cup \{e\}) \setminus \{j\}$ is a not a base for all $j > e$. If an element not in $B$ is not externally active with respect to $B$, we say that it is \newword{externally passive} with respect to $B$. We denote $e_{\M}(B)$ to count the number of such elements.

\begin{lemma}[\cite{Sch},\cite{Wh2}]
\label{lem:lex}
Let $(h_0,\cdots,h_r)$ be the h-vector of a matroid $\M$. For $0 \leq i \leq r$, $h_i$ is the number of bases of $\M$ with $r-i$ internally active elements.  
\end{lemma}

\begin{remark}
\label{rem:dual}
The way we will view $h_i$ in this paper is to count the number of bases in the dual-matroid of $\M$ with $i$ externally passive elements.
\end{remark}

Our main result in this paper is that:

\begin{theorem}
The h-vector of a cotransversal matroid is a pure O-sequence. In other words, Stanley's conjecture is true for cotransversal matroids.
\end{theorem}

Now let's go over the basics of \newword{transversal matroids}. Let $\A$ be a family $(A_1,\ldots,A_r)$ of subsets of the set $L = \{\bar{1},\ldots,\bar{n}\}$. Then the bipartite graph $\G(\A)$ associated with $\A$ has vertex set $L$ and $R = \{1,\ldots,r\}$ with edge set given by $\{(a,b)| a \in L, b \in R $ and $a \in A_b \}$. Throughout the paper, we will call the vertex set $L$ and $R$ of a bipartite graph as the set of \newword{left vertices} and the set of \newword{right vertices} respectively.

Given a subgraph $T$ of this graph, let $lt(T)$ denote the subset of vertices of $L$ covered by edges of $T$ and let $rt(T)$ denote the subset of vertices of $R$ covered by edges of $T$. Then collection of $lt(T)$ for all maximal matchings of $\G(\A)$ form the set of \newword{bases} of a matroid. We denote this matroid by $\M(\A)$. If $\M$ is an arbitrary matroid and $\M \cong \M(\A)$ for some family $\A$ of sets, then we call $\M$ a transversal matroid and $\A$ a \newword{presentation} of $\M$. 

In Figure~\ref{fig:fig1}, we have a presentation of a family $(\{\bar{1},\bar{2},\bar{6},\bar{7},\bar{8},\bar{9}\},\{\bar{3},\bar{4},\bar{5},\bar{6},\bar{7},\bar{8},\bar{9}\})$.

\begin{figure}[htb!]
\centering%
\includegraphics[width=0.25\textwidth]{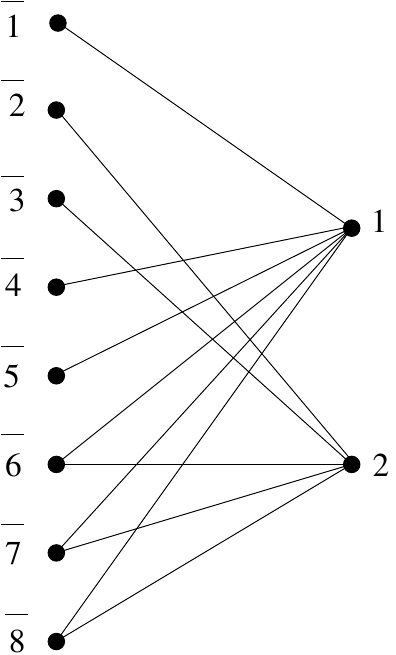}
\caption{Bipartite graph defining $\M$}
\label{fig:fig1}
\end{figure}


The following lemma is given as an exercise in \cite{Ox}.

\begin{lemma}[\cite{Ox}]
\label{lem:nicepres}
Let $\M$ be a transversal matroid that has rank $r$. Then there exists a presentation of $\M$ that has exactly $r$ members.
\end{lemma}





Given a vertex $v$ inside a graph $G$, we will use $N(v)$ to denote the set of neighbors of $v$. It is a well-known fact that a subset $H = \{h_1,\ldots,h_r\}$ of $L$ is a base of $\M(\A)$ if and only if $N(h_1),\ldots,N(h_r)$ satisfies the Hall Marriage condition (if union of any $t$ sets of the collection has cardinality at least $t$) \cite{Ox}.

\begin{theorem}[\cite{Ox}]
Let $I_1,\cdots,I_r \subseteq [r]$. The following conditions are equivalent:
\begin{enumerate}
\item For any $S \subseteq [r]$ we have $|\bigcup_{i \in S} I_i| \geq |S|$.
\item There exists a bijection $f$ from $[r]$ to $[r]$ such that for all $t \in [r]$, $f(t) \in I_t$.
\end{enumerate}
The first condition is called the \newword{hall marriage condition}, and the bijection $f$ in the second condition is referred to as the \newword{system of distinct representatives}.
\end{theorem}

\section{Generalized permutohedra}

In this section, we review generalized permutohedra and study some properties of spanning trees that we will be using in this paper. The content related to generalized permutohedra follows that of \cite{Po}.

\begin{definition}[\cite{Po}]
Let $d$ be the dimension of the Minkowski sum $P_1 + \cdots + P_n$, where $P_1,\ldots,P_n$ are convex polytopes. A \newword{Minkowski cell} in this sum is a polytope $B_1 + \cdots + B_n$ of dimension $d$ where $B_i$ is the convex hull of some subset of vertices of $P_i$. A \newword{mixed subdivision} of the sum is the decomposition into union of Minkowski cells such that intersection of any two cells is their common face. A mixed subdivision is \newword{fine} if for all cells $B_1 + \cdots + B_n$, all $B_i$ are simplices and $\sum dim B_i = d$.
\end{definition}

\begin{remark}
All mixed subdivisions in our paper, unless otherwise stated, will be referring to fine mixed subdivisions. 
\end{remark}

Let $G \subseteq K_{n,r+1}$ be a bipartite graph with no isolated vertices. Label the set of left vertices using $\bar{1},\ldots,\bar{n}$, and label the set of right vertices using $0,\ldots,r$. We will use $\hat{[r]}$ to denote the set $\{0,1,\cdots,r\}$. Let us associate $G$ with the collection $\I_G$ of subsets $I_1,\cdots,I_n \subseteq \hat{[r]}:= \{0,1,\cdots,r\}$ such that $j \in I_i$ if and only if there is an edge in $G$ which connects a left vertex labeled $\bar{i}$ with a right vertex labeled $j$.

 Let $e_0,\ldots,e_r$ be the coordinate vectors of $\RR^{r+1}$. The generalized permutohedron $P_G$ is defined as the Minkowski sum
$$P_G =  \Delta_{I_1} + \cdots + \Delta_{I_n},$$
where $\Delta_I$ is defined to be to be the convex hull of points $e_i$ for $i \in I$.

\begin{remark}[\cite{Po}]
\label{rem:genperm}
This polytope is a special case of the family of polytopes that can be defined by:
$$\{(t_0,\ldots,t_r) \in \RR^{r+1} | \sum_{i=0}^r t_i = z_{\hat{[r]}}, \sum_{i \in I} t_i \geq z_I \}.$$ 
\end{remark}

\begin{proposition}[\cite{Po}]
Let $I_1,\cdots,I_r \subseteq \hat{[r]}$. The following conditions are equivalent:
\begin{enumerate}
\item For any distinct $i_1,\cdots,i_k$, we have $|I_{i_1} \cup \cdots \cup I_{i_k}| \geq k+1$.
\item For any $j \in \hat{[r]}$, there is a system of distinct representatives in $I_1,\cdots,I_{r}$ that avoids $j$.
\end{enumerate}
The above condition is called the \newword{dragon marriage condition}.
\end{proposition}

There is a nice connection between Hall's marriage condition and the dragon marriage condition.

\begin{remark}
\label{rem:halldragon}
When $H_1,\ldots,H_n$ are subsets of $[r]$, they satisfy Hall's marriage condition if and only if $\{0\} \cup H_1,\ldots, \{0\} \cup H_n$ satisfy the dragon marriage condition. 
\end{remark}

\begin{definition}[\cite{Po}]
Let us say that a sequence of nonnegative integers $(a_1,\cdots,a_n)$ is a $G$-\newword{draconian sequence} if $\sum a_i = r$ and, for any subset $\{i_1 < \cdots < i_k\} \subseteq [n]$, we have $|I_{i_1} \cup \cdots \cup I_{i_k}| \geq a_{i_1} + \cdots + a_{i_k}+1$. Equivalently, if the sequence $I_1^{a_1},\cdots,I_n^{a_n}$, where $I^a$ means $I$ repeated $a$ times, satisfies the dragon marriage condition.
\end{definition}

One important property of generalized permutohedra is that fine Minkowski cells can be described by spanning trees of $G$. For a sequence of nonempty subsets $\J = (J_1,\cdots,J_n)$, let $G_{\J}$ be the graph with edges $(\bar{i},j)$ for $j \in J_i$.

\begin{lemma}[\cite{Po}]
Each fine mixed cell in a mixed subdivision of $P_G$ has the form $\Delta_{J_1} + \cdots \Delta_{J_n}$, for some sequence of nonempty subsets $\J = (J_1,\cdots,J_n)$ in $\hat{[r]}$ such that $G_{\J}$ is a spanning tree of $G$.
\end{lemma}

\begin{remark}
\label{rem:facet}
As noted in \cite{Po}, the above lemma implies that each fine mixed cell $\Delta_{J_1} + \cdots \Delta_{J_n}$ is isomorphic to the direct product of simplices $\Delta_{J_1} \times \cdots \times \Delta_{J_n}$. By choosing any $j$ inside $J_i$ with $|J_i| > 1$, the product $\Delta_{J_1} \times \cdots \Delta_{J_i \setminus \{j\}} \times \cdots \Delta_{J_n}$ describes a facet of the cell $\Delta_{J_1} \times \cdots \times \Delta_{J_n}$. And any facet is of such format.
\end{remark}

Given a spanning tree $T$ of $G$, we denote $\prod_T$ to be the corresponding Minkowski cell $\Delta_{J_1} + \cdots + \Delta_{J_n}$. We can say a bit more about the lattice points in each $\prod_T$:

\begin{proposition}[\cite{Po}]
\label{prop:gplatticepoints}
Any integer lattice point of a fine Minkowski cell $\prod_{G_{\J}}$ in $P_G$ is of form $p_1 + \cdots + p_n$ where $p_i$ is an integer lattice point in $\Delta_{J_i}$.
\end{proposition}

Given any subgraph $T$ of $G$, define the \newword{left degree vector} $ld(T)=(d_{\bar{1}},\cdots,d_{\bar{n}})$ and the \newword{right degree vector} $rd(T) = (d_0,\cdots,d_r)$ where $d_{\bar{i}}$ and $d_j$ is the degree of the vertex $\bar{i}$ and $j$ respectively in $T$ minus $1$. The following proposition is stated in the proof of Theorem 11.3 in \cite{Po}.

\begin{proposition}[\cite{Po}]
\label{prop:Gcell}
Let us fix a fine mixed subdivision $\{\prod_{T_1},\cdots,\prod_{T_s}\}$ of the polytope $P_G$. Then the map $\prod_{T_i} \rightarrow ld(T_i)$ is a bijection between fine cells $\prod_{T_i}$ in this subdivision and $G$-draconian sequences.
\end{proposition}

For two spanning trees $T$ and $T'$ of $G$, let $U(T,T')$ be the directed graph which is the union of edges of $T$ and $T'$ with edges of $T$ oriented from left to right and edges of $T'$ oriented from right to left. A directed \newword{cycle} is a sequence of directed edges $(i_1,i_2),(i_2,i_3),\cdots,(i_{k-1},i_k),(i_k,i_1)$ such that all $i_1,\cdots,i_k$ are distinct.

\begin{lemma}[\cite{Po}]
\label{lem:utt}
For two spanning trees $T,T'$, the corresponding Minkowski cells can be in the same mixed subdivision only if $U(T,T')$ has no directed cycles of length $\geq 4$.
\end{lemma}

We will say that $T,T'$ are \newword{compatible} if it satisfies the condition of Lemma~\ref{lem:utt}, and \newword{incompatible} if not.

Before we end, we will state some basic facts about spanning trees of bipartite graphs that we will be using. Recall that we are assuming $G$ to be a bipartite graph inside $K_{n,r+1}$. Let us add in the extra assumption that $n \geq r$. Let $T$ be a spanning tree of $G$. We use $LD_T(I)$ to denote the sum of $d_i$'s for $i \in I$, where $ld(T) = (d_1,\cdots,d_n)$. We use $N_T(I)$ to denote the set of neighbors of vertex set $I$ inside a graph $T$.  

\begin{lemma}
\label{lem:lrtree}
Let $T$ be any spanning tree of $G \subseteq K_{n,r+1}$, where $n \geq r$. Set $I$ to be a subset of left vertices such that $|I| < n$. Then $|N_T(I)| \geq LD_T(I) + 1$. Let $J$ be a subset of right vertices such that $|J| < r+1$. Then $LD_T(N_T(J)) \geq |J|$.
\end{lemma}

\begin{proof}
The first claim follows directly from the fact that $T$ is a spanning tree. For the second claim, consider the induced subgraph $S$ of $T$ by looking at the vertices $J \cup N_T(J) \cup N_T(N_T(J)) = N_T(J) \cup N_T(N_T(J))$. Let $c$ be the number of connected components of $S$, and divide $J$ into $J_1,\ldots,J_c$ such that for each $i$, the union $N_T(J_i) \cup N_T(N_T(J_i))$ is the set of vertices of a connected component $S_i$. For each $i$, we have $|N_T(N_T(J_i))| > |J_i|$ since $T$ is a spanning tree. 

In general, given a spanning tree $T$ of $K_{n,r}$, we have $LD_T([n]) = r-1$ since the total number of edges is $n+r-1$. The component $S_i$ is a spanning tree with left vertex set $N_T(J_i)$ and right vertex set $N_T(N_T(J_i))$. Hence we get $LD_T(N_T(J_i)) = LD_{S_i}(N_T(J_i)) = |N_T(N_T(J_i))| - 1 \geq |J_i|$. Summing the inequalities for all components $S_i$, we get $LD_T(N_T(J)) \geq |J|$.



\end{proof}



\begin{lemma}
\label{lem:infoconn}
Let $T$ and $T'$ be spanning trees of $G \subseteq K_{n,r+1}$, where $n \geq r$. Denote the left-degree vector of $T$ as $(d_1,\cdots,d_n)$ and the left-degree vector of $T'$ as $(d_1',\cdots,d_n')$. If after some relabeling of the set $[n]$,
\begin{itemize}
\item $d_n < d_n'$ and $d_1 > d_1'$,
\item $d_i \geq d_i'$ for all $i \not = n$,
\item $0$ is connected to $\bar{1},\bar{n}$ via an edge in $T'$,
\item $0$ is connected to $\bar{n}$ via an edge in $T$,
\end{itemize}

then $T$ and $T'$ are incompatible.
\end{lemma}

\begin{proof}

Let $H$ be a subset of right vertices of $G$ such that $\bar{n} \not \in N_{T'}(H)$. Using Lemma~\ref{lem:lrtree}, we get $|N_T(N_{T'}(H))| \geq LD_T(N_{T'}(H)) + 1$. And from the way that $T$ and $T'$ was constructed, we get $LD_T(N_{T'}(H)) \geq LD_{T'}(N_{T'}(H))$. By applying Lemma~\ref{lem:lrtree} again, we get $LD_{T'}(N_{T'}(H)) \geq |H|$, from which we can conclude that $|N_T(N_{T'}(H))| > |H|$. Notice that if $N_{T'}(H)$ contains $\bar{1}$, then we have $LD_T(N_{T'}(H)) > LD_{T'}(N_{T'}(H))$, and hence we get $|N_T(N_{T'}(H))| > |H| + 1$.

Assume that $T,T'$ are compatible for the sake of contradiction. We have an edge $(\bar{n},0)$ in $T$ and $(\bar{1},0)$ in $T'$. Denote $H_1$ to be $N_T(\bar{1}) \setminus \{0\}$. To prevent an alternating cycle of length greater than $4$ in $U(T,T')$, we have $\bar{n} \not \in N_{T'}(H_1)$. According to the argument in the previous paragraph, $H_2 := N_T(N_{T'}(H_1)) \setminus \{0\}$ is strictly larger than $H_1$. And again to prevent an alternating cycle of length greater than $4$ in $U(T,T')$, we have $\bar{n} \not \in N_{T'}(H_2)$. By repeating this procedure, setting $H_{i+1}$ to be $N_T(N_{T'}(H_i)) \setminus \{0\}$ in each step, this goes on and on, contradicting the fact that number of vertices in $G$ is finite.
\end{proof}

\section{Lattice points of $P_{\M}$ and bases of $\M$.}

In this section, given a transversal matroid $\M$, we construct a generalized permutohedron $P_{\M}$ from it. Moreover, we show that any fine mixed subdivision of $P_{\M}$ induces a bijection between the bases of $\M$ and lattice points of $P_{\M}$ lying inside the region satisfying $x_i \geq 1$ for all $i \in [r]$.

Let $\M$ be a transversal matroid of rank $r$ over the base set $[\bar{n}] = \{\bar{1},\ldots,\bar{n}\}$. Then by Lemma~\ref{lem:nicepres}, there is a bipartite graph that gives a presentation of $\M$, and is a subgraph of the complete bipartite graph $K_{n,r}$. As before, we label the set of left vertices by $\bar{1},\ldots,\bar{n}$ and label the set of right vertices by $1,\ldots,r$. Now we add a vertex labeled $0$ and connect it to all left vertices to get a new bipartite graph $G$. Then we define $P_{\M}$ to be the generalized permutohedron $P_G$. 

As before, we use $I_1,\ldots,I_n$ to denote $N(\bar{1}),\ldots,N(\bar{n})$. One property to keep an eye on is that $0$ is contained in all of those sets. Our strategy for showing Stanley's conjecture is to assign a bijection between bases of $\M$ and lattice points of $P_{\M}$ that satisfy $x_1,\ldots,x_r \geq 1$. 

Given a generalized permutohedron $P$, let $p$ be an integer lattice point of $P$. We say that $p$ is a \newword{base point} of $P$ if $p + \epsilon \mu$ is a point inside $P$ for a very small positive number $\epsilon$, where $\mu$ is defined to be $ re_0 - \sum_{i=1}^r e_i$.

\begin{lemma}
Base points of $P_{\M}$ are exactly the integer lattice points of $P_{\M}$ that satisfy $x_1,\ldots,x_r \geq 1$.
\end{lemma}
\begin{proof}
All summands of $P_{\M}$ are simplices which contain the vertex $e_0$. Let $p_0$ be the unique vertex of $P_{\M}$ given by the coordinates $(n,0,\ldots,0)$. By Remark~\ref{rem:facet}, each facet surrounding $p_0$ is on a hyperplane $x_i=0$ for some $i \in [r]$. The integer lattice points of $P_{\M}$ that are not base points, are exactly the points on those facets.
\end{proof}

Let $\prod_{\J} = \Delta_{J_1} \times \cdots \times \Delta_{J_n}$ be some Minkowski cell inside a mixed subdivision of $P_{\M}$. 

\begin{definition}
If for all $i \in [n]$, we have $|J_i| \leq 2$, we say that $\prod_{\J}$ is \newword{zonotopal}. 
\end{definition}

It is easy to see that if $\prod_{\J}$ is a zonotopal cell, then $\prod_{\J}$ is actually a zonotope.

\begin{lemma}
\label{lem:zono}
A Minkowski cell $\prod_{\J}$ inside a mixed subdivision of $P_{\M}$ contains a base point of $\prod_{\J}$ if and only if $\prod_{\J}$ is zonotopal.
\end{lemma}

\begin{proof}
We first show that if $\prod_{\J}$ is zonotopal, it contains a base point of $\prod_{\J}$. We construct a subgraph $T$ of $K_{r+1}$ by collecting the edges $(a,b)$ for each $J_i = \{a,b\}$. $T$ is a spanning tree since $G_{\J}$ is a spanning tree of $K_{n,r+1}$. Think of $T$ as a rooted tree having $0$ as the root. For each $J_i = \{a,b\}$, where $b$ is a descendant of $a$, set $q_i$ to be $e_a$ and $p_i$ to be $e_b$. And let $l_i$ denote the number of descendants of $a$ inside $T$. Then consider the point $p = \sum p_i$. We will show that this is a base point of $\prod_{\J}$. For each $i \in [n]$, the point $p_i + l_i \epsilon(q_i - p_i)$ is inside $\Delta_{J_i}$. Therefore, we can conclude that $p + \epsilon \sum l_i (q_i-p_i) = p + \epsilon \mu$ is a point inside $\prod_{\J}$.

We now show that for $\prod_{\J}$ to contain a base point, $\prod_{\J}$ has to be zonotopal. Let $p = p_1 + \cdots + p_n$ be the base point of $\prod_{\J}$, where $p_i \in \Delta_{J_i}$. The point $p$ being a base point implies that we can decrease the value of $a$-th coordinate from $p$ and still stay in $\prod_{\J}$ for all $a \in [r]$. In order for this to be true, for each $a \in [r]$, there has to exist $b \in [n]$ such that we can decrease the value of $a$-th coordinate from $p_b$ and still stay in $\Delta_{J_b}$. But given any $p_b \in \Delta_{J_b}$, there is exactly one coordinate $a \in [r]$ where we can decrease its value and still say in $\Delta_{J_b}$ if $|J_b| \geq 2$, and none otherwise. Therefore, we need at least $r$ $J_i$'s having cardinality $\geq 2$, and this implies that $\prod_{\J}$ is zonotopal.



\end{proof}

\begin{remark}
\label{rem:basecoord}
From the above proof, it is easy to see that the coordinate of the base point is only affected by $J_i$'s such that $|J_i| = 1$. More precisely, the coordinate of the point is given by $(n-r-x_1-\cdots-x_r,x_1+1,\ldots,x_r+1)$, where $x_k$ counts the number of times $k$ appears among $J_i$'s having cardinality $1$.
\end{remark}

\begin{proposition}
\label{prop:bij}
There is a bijection between base points of $P_{\M}$ and bases of $\M$.
\end{proposition}
\begin{proof}
Given a fixed fine mixed subdivision of $P_{\M}$, Proposition~\ref{prop:Gcell} and Remark~\ref{rem:halldragon} tells us that there is a bijection between zonotopal cells of $P_{\M}$ and bases of $\M$. All we need to show is that every base point of $P_{\M}$ is a base point of some zonotopal cell.

The facets of possible cells of $P_{\M}$ are of form $\sum_{i \in I} x_i = z_I$ for some subset $I$ of $\{0\} \cup [r]$. This means that none of the facets are parallel to the vector $\mu$, which implies that $p + \epsilon\mu$ is in the interior of some cell which contains $p$ on its hull. This cell has to be zonotopal by Lemma~\ref{lem:zono}.
\end{proof}

We have seen that each fine mixed subdivision of $P_{\M}$ induces a bijection between base points of $P_{\M}$ and bases of $\M$. In the next section, we come up with a fine mixed subdivision such that $n-r-x_0$ of a base point equals the externally passive degree of the corresponding base in $\M$.



\section{Lexicographical subdivision of $P_{\M}$.}

In this section, we want to find a fine mixed subdivision of $P_{\M}$ such that if we use the bijective map defined in the previous section to associate the bases to the lattice points of $P_{\M}$, the externally passive degree can be read off by looking at the sum of all coordinates except $0$.

We use the fact that the fine mixed subdivision of a generalized permutohedron is related to the triangulations of certain polytopes via the \newword{Cayley trick}. Let $e_{\bar{1}},\ldots,e_{\bar{n}},e_{0},e_{1},\ldots,e_{r}$ be the standard basis of $\RR^{n+r+1}$. Embed the space $\RR^{r+1}$ where the polytopes $\Delta_I$ live for $I \subseteq [\hat{r}]$. As before, let $I_i$ denote the collection $j$'s such that $(\bar{i},j)$ is an edge of $G$. The root polytope $Q_G$ is defined as the convex hull of the vertices $e_{\bar{i}} + e_j$'s, for each edge $(\bar{i},j)$ of $G$.

\begin{Lemma}[\cite{Po}]
Fine mixed subdivisions of $P_G$ are in one-to-one correspondence with triangulations of $Q_G$. A fine mixed cell in $P_G$ given by $\Delta_{J_1} \times \cdots \times \Delta_{J_n}$ corresponds to a simplex which has vertices $e_{\bar{i}} + e_j$ for each pair $(\bar{i},j)$ satisfying $j \in J_i$.
\end{Lemma}

Let $P_G$ be a generalized permutohedron and $P_{G'}$ be $P_G + \Delta_J$. In other words, $G'$ is a bipartite graph obtained by adding a left vertex $v$ with neighborhood $J$ to the bipartite graph $G$. Start from a triangulation of $Q_G$. This naturally induces a triangulation on the cone formed by $e_v + e_0$ inside $Q_{G'}$. This cone is a subpolytope of $Q_{G'}$, so we can extend the triangulation of the cone to a triangulation of $Q_{G'}$. We say that such triangulation of $Q_{G'}$ is obtained by extending the triangulation of $Q_G$ in direction $0$. And for the corresponding mixed subdivisions, we say that the mixed subdivision of $P_{G'}$ is obtained by extending the mixed subdivision of $P_G$ in direction $0$. One can see that for each cell $\prod_{T}$ in the mixed subdivision of $P_G$, $\prod_{T} + \Delta_{\{0\}}$ is a cell inside the extended mixed subdivision of $P_{G'}$. We will use this extension method to define a \newword{lexicographical subdivision} of $P_G$.

Start from a Minkowski sum $\Delta_{\{0,1\}} + \cdots + \Delta_{\{0,r\}}$. We use $X_0$ to denote this polytope. Now use $X_i$ to denote the sum $X_0 + P_i$, where $P_i$ denotes the sum $\Delta_{I_1} + \cdots + \Delta_{I_i}$. We start from a subdivision of $X_0$, which is unique, and repeat the process of extending the subdivision in direction $0$ to obtain a subdivision of $X_i$ for each $i \in [n]$. We call this the lexicographical subdivision of $X_i$.

\begin{lemma}
\label{lem:0con}
Let $\prod_T$ be a cell inside the lexicographical subdivision of $X_i$, for $i \geq 1$. Then $0 \in T_{r+i}$. 
\end{lemma}
\begin{proof}
We will use induction on the size of $|T_{r+i}|$. When $|T_{r+i}| = 1$, the cell $\prod_T$ has left-degree vector $(|T_1|-1,\ldots,|T_{r+i-1}|-1,0)$. Proposition~\ref{prop:Gcell} implies that there is some cell $\prod_{T'}$ that has left degree vector $(|T_1|-1,\ldots,|T_{r+i-1}|-1)$ in $X_{i-1}$. From the definition of lexicographical subdivision, there is a cell corresponding to the tree $(T'_1,\ldots,T'_{r+i-1},\{0\})$ in $X_i$. Since this cell has the same left degree vector as $\prod_T$, Proposition~\ref{prop:Gcell} tells us that $T_{r+i} = \{0\}$.

Now assume for the sake of induction that $0 \in J_{r+i}$ for all cells $\prod_{\J}$ such that $|J_{r+i}| < |T_{r+i}|$. There is some $q \in T_{r+i}$ such that by crossing the facet $\Delta_{T_1} + \cdots + \Delta_{T_{r+i} \setminus \{q\}}$, we reach another cell in $X_i$. If we denote this cell as $\prod_S$, we have $|S_{r+i}| < |T_{r+i}|$ due to Proposition~\ref{prop:Gcell}. And by induction hypothesis, we have $0 \in S_{r+i}$. Since $T_{r+i}  \setminus \{q\} = S_{r+i}$, we get $0 \in T_{r+i}$.
\end{proof}

Let $\prod_{T}$ be a cell inside a lexicographical subdivision of $X_i$, that intersects the region $0 < x_j < 1$ and does not lie inside $X_{i-1} + \Delta_{\{0\}}$. Writing $T = (T_1,\ldots,T_{r+i})$, we can see that $T_j \not = \{j\}$, since otherwise the cell will not intersect with the region $0 < x_j <1$. And by comparing this cell to the cell $\prod_{T'}$ which is given by $\Delta_{\{0,1\}} + \cdots + \Delta_{\{0,r\}} + \Delta_{\{0\}} + \cdots + \Delta_{\{0\}}$, we can see that $j \not \in T_k$ for $k>r$, since otherwise we have a length $4$ alternating cycle in $U(T,T')$. But since $T$ is a spanning tree, some $T_i$ has to include $j$, which implies that $j \in T_j$. Therefore, we can conclude that $T_j = \{0,j\}$ and that the cell $\prod_T$ is included in the region $0 \leq x_j \leq 1$.


\begin{remark}
\label{rem:lexsub}
Inside the lexicographical subdivision of $X_i$, no cell crosses $x_j = 1$ for each $j \in [r]$. Each cell $\prod_{T}$ of $X_i$ inside the region $x_i \geq 1$ for all $i\in [r]$ has $T_i = \{i\}$ for $i \in [r]$.
\end{remark}

Since no cell crosses $x_j = 1$ inside the lexicographical subdivision of $X_i$, we can cut the lexicographical subdivision of $X_i$ via $x_j \geq 1$ for all $j \in [r]$ to get a mixed subdivision of $P_i$. When $i=n$, we call this the lexicographical subdivision of $P_G$.





Now we wish to show that given a lexicographical subdivision of $P_{\M}$, and using the bijection given via Proposition~\ref{prop:bij}, the value $x_1+\cdots+x_r-r$ of a base point equals the externally passive degree of the corresponding base in $\M$. To do this, we first need to define some notations. We will use $H_1,\ldots,H_n$ to denote $N_G(\bar{1}) \setminus \{0\},\ldots,N_G(\bar{n}) \setminus \{0\}$. Given a base $B = \{b_1 < \cdots < b_r \}$ of $\M$, we call the collection of sets $H_{b_1},\ldots,H_{b_r}$ to be the \newword{type sequence} of $B$. Given a collection $\HH^a = \{H_1^{a_1},\ldots,H_n^{a_n}\}$ that satisfies hall's condition, we denote $EP_{\M}(\HH^a)$ to denote the collection of $H_i$'s such that there exists $j<i$ for which $\HH^a \setminus \{H_i\} \cup \{H_j\}$ satisfies the hall's condition. Beware that the collection is considered as a multiset: for example, $\{H_1^2,H_2^0,H_3^1\} \cup \{H_1\} = \{H_1^3,H_2^0,H_3^1\}$.

\begin{remark}
Let the type sequence of $B \in \M$ be $\HH^a$. Then the point corresponding to $B$ via the bijection in Proposition~\ref{prop:bij} is the base point of a cell having left degree vector given by $a$.
\end{remark}

Now we show that given a cell $\prod_{\J}$ with left degree vector $a$ inside the lexicographical subdivision of $P_{\M}$, there is a connection between whether $0$ is in $J_i$ or not and whether $H_i$ is a member of $EP_{\M}(\HH^a)$ or not.

\begin{lemma}
\label{lem:main}
Let $\M$ be a transversal matroid and $P_{\M}$ be its corresponding generalized permutohedron. Look at a Minkowski cell $\prod_{\J}$, where we write $\J = \{J_1,\ldots,J_n\}$, with left degree vector $a$ inside the lexicographical subdivision of $P_{\M}$. We have $0 \not \in J_i$ if and only if $H_i \in EP_{\M}(\HH^a)$.

\end{lemma}

\begin{proof}

We first show that it is enough to show the claim for $X_n$, which was used to define lexicographical subdivision. Let $G'$ be the bipartite graph corresponding to $X_n$. It has left vertices $\bar{1},\ldots,\overline{r+n}$, where each vertex $\bar{i}$ for $i \leq r$ is connected to right vertices $0$ and $i$, and each vertex $\overline{r+i}$ for $i \leq n$ is connected to $I_i$. We set $\M'$ to be the transversal matroid which is represented by $G'$, and set $\KK = (\{1\},\ldots,\{r\},H_1,\ldots,H_n)$. It is easy to see that $H_i \in EP_{\M}(\HH^a)$ if and only if $H_i = K_{i+r} \in EP_{\M}(\KK^{a'})$ where $a' = (0,\ldots,0,a_1,\ldots,a_n)$. Combining this with Remark~\ref{rem:lexsub}, we can conclude that to prove the lemma, it is enough to show for $\M'$ and $X_n$ instead. 

We start with $X_0$. Since the only cell is $X_0$ itself, it is easy to see that the claim holds. For the sake of induction, assume that the claim holds for $X_0,\ldots,X_{q-1}$. This means that the claim holds for cells of $X_q$ with $a_q = 0$. Again, assume for the sake of induction that the claim holds for cells of $X_q$ with left degree vector given by $d$, where $d_q < a_q$. 

 Set $a'$ to be obtained from $a$ by negating $1$ from $a_q$. Given the collection $\HH^{a'}$, use $Q_a$ to denote the largest subset $Q$ of $[q-1]$ such that there exists exactly $|Q|$ subsets of $Q$ inside the collection. Such $Q_a$ is well defined due to the following reasoning: if $A$ and $B$ are such sets that do not contain each other, there are at least $|A \setminus A \cap B|$ number of subsets of $A$ not contained in $B$. So there are at least $|A \cup B|$ number of subsets of $A \cup B$, but this number cannot exceed $|A \cup B|$ since the collection $I^a$ satisfies Hall's marriage condition. Hence if $A$ and $B$ are two sets that satisfy the condition, than $A \cup B$ also satisfies the condition.


For $i < q$, if $H_i \not \subseteq Q_a$, the collection $\HH^a \setminus \{H_q\} \cup \{H_i\}$ satisfies Hall's marriage condition, and hence $H_i \in EP_{\M}(\HH^a)$. To see this, for the sake of contradiction, assume there is some distinct $i,i_1,\ldots,i_s$ such that $|H_i \cup H_{i_1} \cup \cdots \cup H_{i_s}| = s$. This implies that $H_i \subseteq H_i \cup H_{i_1} \cup \cdots \cup H_{i_s}$ and $|H_{i_1} \cup \cdots \cup H_{i_s}| = s$.  It follows that $H_i \subseteq Q_a$, which gives us a contradiction.

Therefore, we get some sequence $b$ such that:
\begin{itemize}
\item $b_q = a_q -1$,
\item $b_i = a_i + 1$,
\item $b_j = a_j$ for $j \not = i,q$,
\item $\HH^b$ satisfies the Hall marriage condition.
\end{itemize}

There exists a cell $\prod_{\CCC}$ with left degree vector $b$ due to Remark~\ref{rem:halldragon} and Proposition~\ref{prop:Gcell}. By induction hypothesis, the claim holds for $\prod_{\CCC}$. Using Lemma~\ref{lem:infoconn}, we get $0 \not \in J_i$ since we have $0 \in J_q,C_q$ due to Lemma~\ref{lem:0con}. Hence we only need to consider $H_i$'s contained in $Q_a$.

By Remark~\ref{rem:facet}, we can cross one of the facets of $\prod_{\J}$ is given by $\Delta_{J_1} + \cdots + \Delta_{J_q \setminus \{i\}}$ to get another cell inside $X_q$. By crossing this facet, we reach a cell $\prod_{\CCC}$ with left degree vector $c$ such that $c_q = a_q -1$. The claim holds for $\prod_{\CCC}$ due to induction hypothesis, and for $\HH^c$ to also satisfy Hall's marriage condition, we need to have $\{H | H \in \HH^a, H \subseteq Q_a\} = \{H | H \in \HH^c, H \subseteq Q_a\}$. This means that:

\begin{itemize}
\item $a_i = c_i$ for all $i$ such that $H_i \subseteq Q_a$,
\item we have $C_i = J_i$ for all $i$ such that $H_i \subseteq Q_a$.
\item for $H_i \subseteq Q_a$, we have $H_i \in EP_{\M}(\HH^a)$ if and only if $H_i \in EP_{\M}(\HH^c)$.
\end{itemize}

Since the lemma holds for $\prod_{\CCC}$, we have also proven the claim for $\prod_{\J}$ in this case too. Hence by induction, we have shown that the claim holds for all cells inside the lexicographical subdivision of $X_q$. And again by induction, we have shown that the statement is true for $X_n$, and from the arugment in the first paragraph, the statement holds for $P_{\M}$.


\end{proof}

Let $B$ be a base in $\M$ and $p = (c_0,\ldots,c_r) \in P_{\M}$ be the corresponding base point via Proposition~\ref{prop:bij}. Combining Remark~\ref{rem:basecoord} and Lemma~\ref{lem:main}, we can see that $c_1+\cdots+c_r-r = e_{\M}(B)$.

For each base point at $(c_0,c_1,\cdots,c_r)$, let's make a monomial ${x_1}^{c_1-1} \cdots {x_r}^{c_r-1}$. Then we get a pure monomial order ideal of which Stanley's conjecture is asking for. 

\begin{proposition}
Let $\M$ be a cotransversal matroid. We denote $\M^{*}$ for the dual matroid, which is in this case a transversal matroid. For each base lattice point $(c_1,\cdots,c_r)$ in $P_{\M^{*}}$, take a monomial ${x_1}^{c_1-1} \cdots {x_r}^{c_r-1}$ to form a collection $X$. Then $X$ is a pure monomial order ideal and its degree sequence equals the h-vector of $\M$.
\end{proposition}
\begin{proof}
We first show that $X$ is a monomial order ideal. Let $(c_0,\ldots,c_r)$ be a point in $P = P_{\M^{*}}$. Let $G$ be the corresponding bipartite graph of $P$. Now consider a subgraph we obtain by deleting the right vertex $i$. This gives us a polytope, with one less dimension, and contains $(c_0,\ldots,0,\ldots,c_r)$ as a point, which is obtained from $(c_0,\ldots,c_r)$ by setting $c_i$ to $0$. This point is also inside $P$, so this implies that $(c_0,\ldots,c_i-1,\ldots,c_r)$ is also inside $P$. This proves that $X$ is a monomial order ideal.

Now let's show that $X$ is pure. Recall that by Proposition~\ref{prop:gplatticepoints}, each lattice point of $P$ is of form $p_1 + \cdots + p_n$, where $p_i$ is a lattice point of $\Delta_{I_i}$. The point $p_1 + \cdots + p_n$ corresponds to a maximal monomial if and only if $p_i \not = e_0$ for all $i \in [n]$. This implies that all such points are on the hyperplane $x_1 + \cdots + x_r = n$, from which we can conclude that the corresponding monomials have the same degree.


\end{proof}

This implies Stanley's conjecture for cotransversal matroids.



\begin{figure}[htb!]
\centering%
\includegraphics[width=0.3\textwidth]{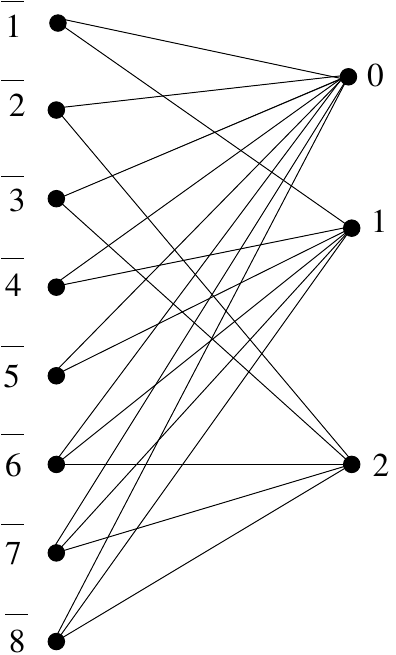}
\caption{Padding of the graph given in Figure~\ref{fig:fig1}.}
\label{fig:fig2}
\end{figure}

Let's look at an example.  We look at a transversal matroid $\M$ given by the bipartite graph in Figure~\ref{fig:fig1}. The padded bipartite graph is given Figure~\ref{fig:fig2}, and we construct a generalized permutohedron from it. For convenience, we will project down to $x_0=0$ to draw the polytope in the $x_1,x_2$-plane. 

First lets look at the cell that lies on the southwest corner. The corresponding summand is given by $\Delta_{0,1} + \Delta_{0,2} + \Delta_{0} + \cdots + \Delta_{0}$. The left-degree vector is given by $(1,1,0,0,0,0,0,0)$ and our bijection assigns the base point of this cell to the base $\{\bar{1},\bar{2}\}$. 

Now consider the leftmost triangle. The corresponding summand is given by $\Delta_{1} + \Delta_{2} + \Delta_{2} + \Delta_{1} + \Delta_{1} + \Delta_{0,1,2} + \Delta_{0} + \Delta_{0}$. This cell is not zonotopal, and there is no base assigned to the cell. If we look at the cell to the top of it, the summand is given by $\Delta_{1} + \Delta_{2} + \Delta_{2} + \Delta_{1} + \Delta_{1} + \Delta_{1,2} + \Delta_{0,2} + \Delta_{0}$. The left-degree vector is given by $(0,0,0,0,0,1,1,0)$, and the base point of the cell is assigned to the base $\{\bar{6},\bar{7}\}$.

\begin{figure}
\centering
\includegraphics[width=0.4\textwidth]{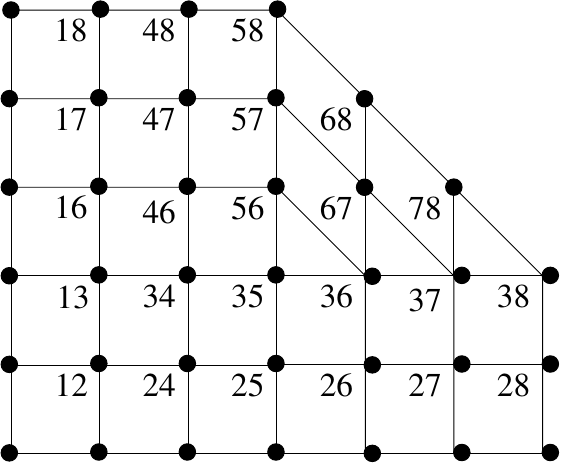}%
\qquad
\begin{tabular}{|c|c|c|}
	\hline
$B$ &  Externally passive elements  \\
	\hline
12   &  $\emptyset$\\
13   &  2\\
16   &  2,3\\
17   &  2,3,6\\
18   &  2,3,6,7\\
24   &  1\\
25   &  1,4\\
26   &  1,4,5 \\
27   &  1,4,5,6\\
28   &  1,4,5,6,7\\
34   & 1,2\\
35   & 1,2,4\\
36   & 1,2,4,5\\
37   & 1,2,4,5,6\\
38   & 1,2,4,5,6,7  \\  
46   & 1,2,3 \\
47   & 1,2,3,6\\
48   & 1,2,3,6,7\\
56   & 1,2,3,4\\
57   & 1,2,3,4,6\\
58   & 1,2,3,4,6,7\\
67   & 1,2,3,4,5\\
68   & 1,2,3,4,5,7\\
78   & 1,2,3,4,5,6\\
\hline
\end{tabular}
\captionlistentry[table]{transversalhedron}
\captionsetup{labelformat=andtable}
\caption{A lexicographical subdivision of $P_{\M}$ and the table of $B \in \M$, where the bars of the base set $\{\bar{1},\ldots,\bar{n}\}$ is omitted for convenience.}
\label{fig:transversalhedron}
\end{figure}


\begin{thebibliography}{llll}
\bibitem{BMMNZ} M. Boij, J. Migliore, R. Mir\`o-Roig, U. Nagel, and
  F. Zanello: \lq \lq On the shape of a pure $O$-sequence'', preprint
  (76 pp.). arXiv:1003.3825.
  
  \bibitem{Ch} M.K. Chari: \emph{Matroid inequalities}, Discrete
  Math. \textbf{147} (1995), 283-286.
\bibitem{Ch2} M.K. Chari: \emph{Two decompositions in topological
  combinatorics with applications to matroid complexes},
  Trans. Amer. Math. Soc. \textbf{349} (1997), no. 10, 3925-3943.
\bibitem{HS} T. Hausel and B. Sturmfels: \emph{Toric hyperK\"aler
  varieties}, Doc. Math. \textbf{7} (2002), 495-534.
\bibitem{Hi} T. Hibi: \emph{What can be said about pure
  $O$-sequences?}, J. Combin. Theory Ser. A \textbf{50} (1989), no. 2,
  319-322.
\bibitem{Me} C. Merino: \emph{The chip firing game and matroid
  complexes}, Discrete models: combinatorics, computation, and
  geometry (2001), Discrete Math. Theor. Comput. Sci. Proc., AA,
  Maison Inform. Math. Discrete. (MIMD), Paris (2001), 245--255.
\bibitem{Ox} J.G. Oxley: ``Matroid theory'', Oxford University Press
  (2006).
\bibitem{Pr} N. Proudfoot: \emph{On the $h$-vector of a matroid
  complex}, unpublished note (2002).
\bibitem{Po} A. Postnikov: \emph{Permutohedra, associahedra, and beyond}, preprint. arXiv:0507163.
\bibitem{Sch} J. Schweig: \emph{On the $h$-Vector of a Lattice Path
  Matroid}, Electron. J. Combin. \textbf{17} (2010), no. 1, N3.
\bibitem{St3} R. Stanley: ``Combinatorics and commutative algebra'',
  Second Ed., Progress in Mathematics \textbf{41}, Birkh\"auser Boston,
  Inc., Boston, MA (1996).
\bibitem{Sto} E. Stokes: ``The $h$-vectors of matroids and the
  arithmetic degree of squarefree strongly stable ideals'',
  Ph.D. Thesis, University of Kentucky (2008).
\bibitem{TSZ} H. T\'ai H\'a, E. Stokes, F.Zanello:  ``Pure O-sequences and matroid h-vectors'', Annal of Combinatorics, to appear.
\bibitem{Wh} N. White, Ed.: ``Theory of Matroids'', Encyclopedia of
  Mathematics and Its Applications \textbf{26}, Cambridge Univ. Press,
  Cambridge (1986).
\bibitem{Wh2} N. White, Ed.: ``Matroids Applications'', Encyclopedia
  of Mathematics and Its Applications \textbf{40}, Cambridge Univ. Press, Cambridge
  (1992).\end{thebibliography}
\end{document}